\documentclass[12pt,twoside]{amsart}
\usepackage{amsmath}
\usepackage{amsthm}
\usepackage{amsfonts}
\usepackage{amssymb}
\usepackage{latexsym}
\usepackage{mathrsfs}
\usepackage{amsmath}
\usepackage{amsthm}
\usepackage{amsfonts}
\usepackage{amssymb}
\usepackage{latexsym}
\usepackage{geometry}
\usepackage{dsfont}
\usepackage[dvips]{graphicx}
\usepackage{color}
\usepackage[all]{xy}

\date{}
\allowdisplaybreaks[4] \footskip=15pt
\renewcommand{\uppercasenonmath}[1]{}

\usepackage{graphicx,amssymb}
\usepackage[all]{xy}
\usepackage{amsmath}

\allowdisplaybreaks
\usepackage{amsthm}
\usepackage{color}

\theoremstyle{plain}
\newtheorem{theorem}{Theorem}[section]
\newtheorem{proposition}[theorem]{Proposition}
\newtheorem{lemma}[theorem]{Lemma}
\newtheorem{corollary}[theorem]{Corollary}
\theoremstyle{definition}
\newtheorem{example}[theorem]{Example}
\newtheorem{definition}[theorem]{Definition}

\theoremstyle{definition}
\newtheorem*{acknowledgement}{Acknowledgement}

\theoremstyle{remark}
\newtheorem{remark}[theorem]{Remark}

\newcommand{\pf}{\noindent\begin {proof}}
\newcommand{\epf}{\end{proof}}

\newcommand{\Ext}{\mbox{\rm Ext}}
\newcommand{\Hom}{\mbox{\rm Hom}}
\newcommand{\Tor}{\mbox{\rm Tor}}

\newcommand{\Prufer}{Pr\"{u}fer}

\def\m{{\frak m}}

\def\p{{\frak p}}

\def\fd{{\rm fd}}

\def\Hom{{\rm Hom}}
\def\Ext{{\rm Ext}}
\def\Tor{{\rm Tor}}

\def\fkm{{\frak m}}

\def\Im{{\rm Im}}

\def\Nil{{\rm Nil}}
\def\NN{{\rm NN}}

\def\NP{{\rm NP}}
\def\Z{{\rm Z}}
\def\U{{\rm U}}
\def\T{{\rm T}}

\def\ZN{{\rm ZN}}

\def\Krull{{\rm Krull}}

\def\Prufer{{\rm Pr\"{u}fer}}

\begin{document}
\begin{center}
{\large  \bf Some Remarks on  $\phi$-Dedekind rings and $\phi$-\Prufer\ rings}

\vspace{0.5cm}   Xiaolei Zhang$^{a}$\ \ \ \ Wei Qi$^a$\\

{\footnotesize a.\ School of Mathematics and Statistics, Shandong University of Technology, Zibo 255049, China\\

Corresponding author: Xiaolei Zhang, E-mail: zxlrghj@163.com\\}
\end{center}

\bigskip
\centerline { \bf  Abstract}
\bigskip
\leftskip10truemm \rightskip10truemm \noindent

In this paper,   the notions of nonnil-injective modules and nonnil-FP-injective modules are  introduced and studied.  Especially, we show that a $\phi$-ring $R$ is an integral domain if and only if any nonnil-injective (resp., nonnil-FP-injective) module $R$-module is injective (resp., FP-injective). Some new characterizations of $\phi$-von Neumann regular rings, nonnil-Notherian rings and nonnil-coherent rings are given.   We finally characterize $\phi$-Dedekind rings and  $\phi$-\Prufer\ rings  in terms of $\phi$-flat modules, nonnil-injective modules and nonnil-FP-injective modules.
\vbox to 0.3cm{}

\leftskip0truemm \rightskip0truemm
\bigskip

Recall from \cite{A97} that a commutative ring $R$ is  an \emph{$\NP$-ring} if the nilpotent radical $\Nil(R)$ is a prime ideal, and a \emph{$\ZN$-ring} if $\Z(R)=\Nil(R)$ where $\Z(R)$ is the set of all zero-divisors of $R$. A prime ideal $P$ of $R$ is called \emph{divided prime} if $P\subsetneq (x)$, for every $x\in R-P$. Set $\mathcal{H}=\{R|R$ is a commutative ring and \Nil(R)\ is a divided prime ideal of $R\}$. A ring $R$ is a \emph{$\phi$-ring} if $R\in \mathcal{H}$. Moreover, a $\ZN$ $\phi$-ring is said to be a \emph{strong $\phi$-ring}. Denote by $\T(R)$ the localization of $R$ at the set of all regular elements. For a  $\phi$-ring $R$, there is a ring homomorphism $\phi:\T(R)\rightarrow R_{\Nil(R)}$ such that $\phi(a/b)=a/b$. Denote by the ring $\phi(R)$ the image of $\phi$ restricted to $R$. In 2001, Badawi \cite{A01} investigated  \emph{$\phi$-chained rings} ($\phi$-CRs for short) which are $\phi$-rings $R$ such that for every $x, y \in R - \Nil(R)$ either $x | y$ or $y | x$. In 2004, Anderson and Badawi \cite{FA04} extended the notion of \Prufer\ domains to that of  \emph{ $\phi$-Pr\"{u}fer rings} which are $\phi$-rings $R$ satisfies that each finitely generated nonnil ideal is $\phi$-invertible.  The authors in \cite{FA04} characterized $\phi$-\Prufer\ rings from the perspective of ring structures, which says that a  $\phi$-ring $R$ is $\phi$-Pr\"{u}fer if and only if $R_{\fkm}$ is a $\phi$-chained ring for any maximal ideal $\fkm$ of $R$ if and only if $R/\Nil(R)$ is a \Prufer\ domain if and only if $\phi(R)$ is \Prufer.
Later in 2005, the authors in \cite{FA05} generalized the concepts of Dedekind domains to the context of rings that are in the class $\mathcal{H}$.  A $\phi$-ring is called a \emph{$\phi$-Dedekind ring} provided that any nonnil ideal is $\phi$-invertible. They also showed that  a  $\phi$-ring $R$ is $\phi$-Dedekind if and only if $R$ is nonnil-Noetherian and   $R_{\fkm}$ is a discrete $\phi$-chained ring for any maximal ideal $\fkm$ of $R$, if and only if  $R$ is nonnil-Noetherian, $\phi$-integral closed and of Krull dimension $\leq 1$, if and only if $R/\Nil(R)$ is a Dedekind domain. Some  generalizations of  Noetherian domains, coherent domains, Bezout domains and Krull domains to the context of rings that are in the class $\mathcal{H}$ are also introduced and studied (see \cite{FA04,FA05,aa16,A03,ALT06}).

The module-theoretic studies of rings in $\mathcal{H}$  started more than a decade ago. In 2006, Yang \cite{Y06} introduced nonnil-injective modules by replacing the ideals in Baer's criterion for injective modules with nonnil ideals, and obtained that a $\phi$-ring $R$ is nonnil-Noetherian if and only if any direct sum of nonnil-injective modules is nonnil-injective.
In 2013, Zhao et al. \cite{ZWT13} introduced and studied the conceptions of \emph{$\phi$-von Neumann rings} which can be defined as the following characterizations: a $\phi$-ring $R$ is  $\phi$-von Neumann if and only if its \Krull\ dimension is $0$, if and only if  any $R$-module is $\phi$-flat, if and only if $R/\Nil(R)$ is a von Neumann regular ring.  In 2018, Zhao \cite{Z18} gave a homological characterization of $\phi$-\Prufer\ rings: a strong $\phi$-ring $R$ is $\phi$-\Prufer\ if and only if each submodule of a $\phi$-flat module is $\phi$-flat, if and only if each nonnil ideal of $R$ is $\phi$-flat.

The main motivation of this paper is to give some characterizations of $\phi$-Dedekind rings and $\phi$-\Prufer\ rings in terms of some new versions of injective modules and FP-injective modules. We first introduce and study the notions of nonnil-injective modules and  nonnil-FP-injective modules, and show that a $\phi$-ring $R$ is an integral domain if and only if any nonnil-injective module $R$-module is injective, if and only if any nonnil-FP-injective module $R$-module is FP-injective (see Theorem \ref{asfap-int}). Some new characterizations of $\phi$-von Neumann regular rings, nonnil-Noetherian rings and nonnil-coherent rings  in terms of $\phi$-flat modules, nonnil-injective modules and nonnil-FP-injective modules are also given (see Theorem \ref{asfap-vn}, Proposition \ref{asfap-nn} and Proposition \ref{asfap-nc} respectively). We  obtain that a strong $\phi$-ring  $R$ is a $\phi$-Dedekind ring if and only if any  divisible module is nonnil-injective, if and only if any $h$-divisible module is nonnil-injective, if and only if  any nonnil ideal of $R$ is projective (see Theorem \ref{asfap}). We also  obtain that a strong $\phi$-ring  $R$ is  $\phi$-\Prufer, if and only if   any  divisible module is nonnil-FP-injective, if and only if any finitely generated nonnil ideal of $R$ is projective, if and only if any ideal of $R$ is $\phi$-flat, if and only if any $R$-module has an epimorphism $\phi$-flat envelope (see Theorem \ref{asfap-prufer}).

\section{nonnil-injective modules and nonnil-FP-injective modules}

Throughout this paper, $R$ denotes an $\NP$-ring with identity and all modules are unitary. We say an ideal $I$ of $R$ is \emph{nonnil} if there exists a non-nilpotent element in  $I$. Denote by $\NN(R)$ the set of all nonnil ideals of $R$.  It is easy to verify that $\NN(R)$ is a multiplicative system of ideals. That is, $R\in \NN(R)$ and  $IJ\in \NN(R)$ for any $I$ and $J$ both in $\NN(R)$. Let $M$ be an $R$-module.  Define
\begin{center}
$\phi$-$tor(M)=\{x\in M|Ix=0$ for some  $I\in \NN(R)\}$.
\end{center}
An $R$-module $M$ is said to be \emph{$\phi$-torsion} (resp., \emph{$\phi$-torsion free}) provided that  $\phi$-$tor(M)=M$ (resp., $\phi$-$tor(M)=0$). Clearly, the class of $\phi$-torsion modules is closed under submodules, quotients, direct sums and direct limits.  Thus an $\NP$-ring $R$ is $\phi$-torsion free if and only if every flat module is $\phi$-torsion free if and only if $R$ is a  $\ZN$-ring (see \cite[Proposition 2.2]{Z18}). The classes of $\phi$-torsion modules and $\phi$-torsion free modules constitute a hereditary torsion theory of finite type. Recall that  an ideal $I$ of $R$ is regular if there exists a regular element (i.e., non-zero-divisor) in $I$.
\begin{lemma}\label{nonnil-regu}
Let $R$ be a $\phi$-ring and $I$ an  ideal of $R$. Then the following assertions are equivalent:
\begin{enumerate}
    \item $I$ is a nonnil ideal of  $R$;
\item  $I/\Nil(R)$ is a nonzero ideal of  $R/\Nil(R)$;
\item $\phi(I)$ is a regular ideal of   $\phi(R)$;
\end{enumerate}
\end{lemma}
\begin{proof} $(1)\Leftrightarrow (2)$: Obvious.

$(1)\Rightarrow (3)$: Let $s$ be a non-nilpotent element in  $I$. Then $\frac{s}{1}\in \phi(I)$ is regular in $\phi(R)$. Indeed, suppose
$\frac{s}{1}\frac{t}{1}=0$ in $\phi(R)$, then there exists a non-nilpotent element $u\in R$ such that $ust=0$. Since $R$ is  a $\phi$-ring,
$us$ is non-nilpotent. Thus $\frac{t}{1}=0$ in $\phi(R)$.

$(3)\Leftarrow(1)$: Let $\frac{s}{1}$ be an regular element in $\phi(I)$ with $s\in I$. Then $s$ is non-nilpotent. Indeed, if $s^n=0$ in $R$,
then $(\frac{s}{1})^n=\frac{s^n}{1}=0$ in  $\phi(R)$ which implies $\frac{s}{1}$ is not regular in $\phi(R)$.
\end{proof}

Recall that an $R$-module $M$ is \emph{injective} (resp., \emph{FP-injective}) if $\Ext_R^1(N,M)=0$ for any (resp., finitely presented) $R$-module $N$. Now we investigate the notions of nonnil-injective modules and nonnil-FP-injective modules using $\phi$-torsion modules.

\begin{definition} Let $R$ be an $\NP$-ring and $M$ an $R$-module.
\begin{enumerate}
    \item  $M$ is called  \emph{nonnil-injective} provided that $\Ext_R^1(T,M)=0$ for any  $\phi$-torsion module $T$.
     \item  $M$ is called  \emph{nonnil-FP-injective} provided that $\Ext_R^1(T,M)=0$ for  any finitely presented $\phi$-torsion module $T$.
\end{enumerate}
\end{definition}

Certainly, an $R$-module $M$ is nonnil-injective if and only if  $\Ext_R^1(R/I,M)=0$ for any  any nonnil ideal $I$ of $R$ (see \cite[Theorem 1.7]{ZZ19}). The class of nonnil-injective modules is closed under  direct summands, direct products and extensions, and the class of nonnil-FP-injective modules is closed under pure submodules, direct sums, direct products and extensions.

Recall from \cite{ZWT13} that an $R$-module $M$ is \emph{$\phi$-flat} if $\Tor^R_1(T,M)=0$ for  any  $\phi$-torsion module $T$. It is well-known that an $R$-module $M$ is $\phi$-flat if and only if $\Tor^R_1(R/I,M)=0$ for any (finitely generated) nonnil ideal $I$ of $R$ (see \cite[Theorem 3.2]{ZWT13}).
\begin{proposition}\label{flat-FP-injective}
Let $R$ be an $\NP$-ring, then the following assertions are equivalent:
\begin{enumerate}
    \item $M$ is $\phi$-flat;
      \item   $\Hom_R(M,E)$ is nonnil-injective for any injective module $E$;
     \item   $\Hom_R(M,E)$ is nonnil-FP-injective for any injective module $E$;
    \item  if $E$ is an injective cogenerator, then $\Hom_R(M,E)$ is nonnil-injective.
       \item  if $E$ is an injective cogenerator, then $\Hom_R(M,E)$ is nonnil-FP-injective.
\end{enumerate}
\end{proposition}
\begin{proof}
$(1)\Rightarrow (2)$: Let $T$ be a  $\phi$-torsion $R$-module and $E$ an injective $R$-module. Since $M$ is $\phi$-flat, $\Ext_R^1(T,\Hom_R(M,E))\cong\Hom_R(\Tor_1^R(T,M),E)=0$.  Thus $\Hom_R(M,E)$ is  nonnil-injective.

$(2)\Rightarrow (3)\Rightarrow (5)$: Trivial.

$(2)\Rightarrow (4)\Rightarrow (5)$: Trivial.

$(5)\Rightarrow (1)$:  Let $I$ be a finitely generated nonnil ideal of $R$ and $E$ an injective cogenerator. Since $\Hom_R(M,E)$ is nonnil-FP-injective,  $\Hom_R(\Tor_1^R(R/I,M),E)\cong \Ext_R^1(R/I,\Hom_R(M,E))=0$. Since $E$ is an injective cogenerator, $\Tor_1^R(R/I,M)=0$. Thus $M$ is $\phi$-flat.
\end{proof}

\begin{proposition}\label{R-nil}
Let $R$ be a $\phi$-ring and $E$  an  $R/\Nil(R)$-module. Then  $E$  is injective  over $R/\Nil(R)$ if and only if $E$ is nonnil-injective over $R$.
\end{proposition}
\begin{proof} Let $I$ be a nonnil ideal of $R$. Set $\overline{R}=R/\Nil(R)$ and $\overline{I}=I/\Nil(R)$. Let  $E$ be an $\overline{R}$-module. The short exact sequence $0\rightarrow I \rightarrow R \rightarrow R/I\rightarrow 0$ induces the long exact sequence of $R$-modules: $$ 0\rightarrow\Hom_R(R/I,E)\rightarrow \Hom_{R}(R, E)\rightarrow \Hom_{R}(I,E) \rightarrow \Ext_R^1(R/I,E)\rightarrow 0.\ \ \ \ \ \ \ \ (a)$$
The short exact sequence $0\rightarrow \overline{I} \rightarrow \overline{R}\rightarrow R/I\rightarrow 0$
induces the long exact sequence of $\overline{R}$-modules:
$$0\rightarrow\Hom_{\overline{R}}(R/I,E)\rightarrow \Hom_{\overline{R}}(\overline{R},E)\rightarrow \Hom_{\overline{R}}( \overline{I},E) \rightarrow \Ext_{\overline{R}}^1(R/I,E) \rightarrow  0.\ \ \ \ \ \ \ \ (b)$$
By \cite[Lemma 1.6]{ZxlZ20}, $I\Nil(R)=\Nil(R)$. Thus $ I\otimes_R \overline{R}\cong I/I\Nil(R)\cong \overline{I}$. Consequently, we have $\Hom_{\overline{R}}( \overline{I},E) \cong \Hom_{\overline{R}}( I\otimes_R \overline{R},E)\cong  \Hom_{R}(I,\Hom_{\overline{R}}(\overline{R},E))\cong \Hom_{R}(I,E)$ by the Adjoint Isomorphism Theorem (see \cite[Theorem 2.2.16]{fk16}). Combining $(a)$ and $(b)$, we have  $E$  is injective  over $R/\Nil(R)$ if and only if $E$ is nonnil-injective  over $R$ (see Lemma \ref{nonnil-regu} and \cite[Lemma 2.4]{FA04}).
\end{proof}

\begin{proposition}\label{R-nil-1}
Let $R$ be a $\phi$-ring and $M$ an FP-injective $R/\Nil(R)$-module. Then $M$ is nonnil-FP-injective over $R$.
\end{proposition}
\begin{proof} Let $T$ be a finitely presented $\phi$-torsion module over $R$. Then there is a short exact sequence $0\rightarrow K\rightarrow F\rightarrow T\rightarrow 0$ where $F$ is a finitely generated free $R$-module and $K$ is finitely generated $R$-module. Set  $\overline{R}=R/\Nil(R)$. By tensoring $\overline{R}$ over $R$, we obtain a long exact sequence $\Tor_1^R(T,\overline{R})\rightarrow K\otimes_R\overline{R} \rightarrow F\otimes_R\overline{R} \rightarrow T\otimes_R\overline{R} \rightarrow 0$ over $\overline{R}$. By \cite[Proposition 1.7]{ZxlZ20}, $\overline{R}$ is $\phi$-flat over $R$ thus $\Tor_1^R(T,\overline{R})=0$. It follows that $T\otimes_R\overline{R}$ is a finitely presented $\overline{R}$-module. There exists a commutative diagram of exact rows as follows:
$$\xymatrix@R=25pt@C=20pt{
  \ar[r]^{} &\Hom_R(F,M)  \ar[r]^{}\ar[d]^{\cong}& \Hom_R(K,M)   \ar[r]^{}\ar[d]^{\cong}&\Ext_R^1(T,M) \ar[r]^{}\ar[d]^{f}&0\\
   \ar[r]^{} &\Hom_{\overline{R}}(F\otimes_R\overline{R},M)\ar[r]^{} &\Hom_{\overline{R}}(K\otimes_R\overline{R},M) \ar[r]^{} &\Ext^1_{\overline{R}}(T\otimes_R\overline{R},M) \ar[r]^{} &0.
}$$
By the Adjoint isomorphism, the left two homomorphisms are isomorphisms. It follows from the Five Lemma that $f$ is also an isomorphism. Since   $M$ is FP-injective  over $\overline{R}$, $\Ext^1_{\overline{R}}(T\otimes_R\overline{R},M)=0$. Then $\Ext_R^1(T,M)=0$. Thus $M$ is nonnil-FP-injective over $R$.
\end{proof}

Obviously, any  FP-injective module is nonnil-FP-injective, and any  injective module is nonnil-injective. However, the converses characterize integral domains.
\begin{theorem}\label{asfap-int}
Let $R$ be  a $\phi$-ring.  Then the following  assertions are equivalenet:
\begin{enumerate}
    \item  $R$ is an integral domain;
    \item   any  nonnil-injective module is injective;
       \item   any  nonnil-FP-injective module is FP-injective.
\end{enumerate}
\end{theorem}
\begin{proof} $(1)\Rightarrow(2)$ and $(1)\Rightarrow(3)$ : Trivial.

 $(2)\Rightarrow(1)$: By \cite[Theorem 3.1.6]{EJ00}, $\Hom_{\mathbb{Z}}(R/\Nil(R),\mathbb{Q}/\mathbb{Z})$ is an injective $R/\Nil(R)$-module. Thus by Proposition \ref{R-nil}, $\Hom_{\mathbb{Z}}(R/\Nil(R),\mathbb{Q}/\mathbb{Z})$  is a nonnil-injective $R$-module, and thus an injective $R$-module. By \cite[Theorem 3.2.10]{EJ00}, $R/\Nil(R)$ is a flat $R$-module.  Let $K$ be a finitely generated nilpotent ideal, then $K\subseteq \Nil(R)\subseteq Rad(R)$. Thus $K/K\Nil(R)=\frac{K\cap \Nil(R)}{K\Nil(R)}= \Tor^R_1(R/K,R/\Nil(R))=0$. It follows from the Nakayama Lemma that $K=0$. Thus  $\Nil(R)=0$, and then $R$ is an integral domain.

  $(3)\Rightarrow(1)$: Similar with  $(2)\Rightarrow(1)$.
\end{proof}

Recall from \cite{ZWT13} that a $\phi$-ring $R$ is said to be $\phi$-von Neumann if the Krull dimension of $R$ is $0$.  It is well known that a $\phi$-ring $R$ is $\phi$-von Neumann if and only if $R/\Nil(R)$ is a von Neumann ring, if and only if any $R$-module is  $\phi$-flat (see \cite[Theorem 4.1]{ZWT13}).
\begin{theorem}\label{asfap-vn}
Let $R$ be a $\phi$-ring.  Then the following  assertions are equivalenet:
\begin{enumerate}
    \item  $R$ is a $\phi$-von Neumann regular ring;
    \item  $R/\Nil(R)$ is a field;
      \item  any non-nilpotent element  in $R$ is invertible.
    \item   any $R$-module is  $\phi$-flat;
       \item   any $R$-module is  nonnil-FP-injective.
          \item   any $R$-module is  nonnil-injective.
\end{enumerate}
\end{theorem}

\begin{proof}
$(1)\Leftrightarrow(4)$: See \cite[Theorem 4.1]{ZWT13}.

$(1)\Rightarrow(2)$: Since $\Nil(R)$ is a prime ideal of $R$, $R/\Nil(R)$ is a $0$-dimensional domain, thus a field by \cite[Theorem 3.1]{H88}.

$(2)\Rightarrow(3)$: Let $a$ be a non-nilpotent element  in $R$. Since $R/\Nil(R)$ is a field, there exists $b\in R$ such that $1-ab\in \Nil(R)$. That is, $(1-ab)^n=0$ for some $n$. It is easy to verify that $a$ is invertible.

$(3)\Rightarrow(2) \Rightarrow(1)$: Trivial.

$(3)\Rightarrow(5)$:  It follows from (3) that  the only  nonnil ideal of $R$ is $R$ itself. Let $T$ be a finitely presented $\phi$-torsion module. Then $T=\phi$-$tor(T)=\{x\in T|Ix=0$ for some nonnil ideal $I$ of  $R \}=0$. It follows that $\Ext_R^1(T,M)=0$. Consequently, $M$ is nonnil-FP-injective.

$(5)\Rightarrow(1)$: Let $I$ be a finitely generated nonnil ideal of $R$. Since for any $R$-module $M$, $\Ext_R^1(R/I,M)=0$ by $(5)$, then $R/I$ is projective. Thus $I$ is an idempotent ideal of $R$. By \cite[Proposition 1.10]{FS01}, $I$ is generated by an idempotent $e\in R$. Thus $R$ is a $\phi$-von Neumann regular ring by \cite[Theorem 4.1]{ZWT13}.

$(3)\Rightarrow(6)$ and $(6)\Rightarrow(5)$: Obvious.
\end{proof}

Recall from \cite{A03} that a $\phi$-ring $R$ is called  \emph{nonnil-Noetherian} if any nonnil ideal of $R$ is finitely generated.
\begin{proposition}\label{asfap-nn}
Let $R$ be a $\phi$-ring.  Then $R$ is nonnil-Noetherian if and only if any nonnil-FP-injective module is nonnil-injective.
\end{proposition}

\begin{proof} Suppose $R$ is a nonnil-Noetherian ring.  Let $I$ be a nonnil ideal of $R$ and $M$  a nonnil-FP-injective module. Then $I$ is  finitely generated, and thus $R/I$ is finitely presented $\phi$-torsion. It follows that  $\Ext^1_R(R/I,M)$. Consequently, $M$ is nonnil-injective by \cite[Theorem 1.7]{ZZ19}. On the other hand, since the class of nonnil-FP-injective modules is closed under direct sums, $R$ is a nonnil-Noetherian ring by \cite[Theorem 1.9]{Y06}
\end{proof}

Recall from \cite{aa16} that a $\phi$-ring $R$ is called \emph{nonnil-coherent} if any finitely generated nonnil ideal of $R$ is finitely presented.   A $\phi$-ring $R$ is nonnil-coherent if and only if any direct product of $\phi$-flat modules is $\phi$-flat, if and only if $R^I$ is $\phi$-flat for any indexing set $I$ (see \cite[Theorem 2.4]{aa16}). Now we give a new characterization of nonnil-coherent rings utilizing the  preenveloping properties of $\phi$-flat modules.
\begin{proposition}\label{asfap-nc}
Let $R$ be a $\phi$-ring.  Then  $R$ is  nonnil-coherent if an only if the class of $\phi$-flat modules is preenveloping.
\end{proposition}

\begin{proof}

Suppose $R$ is a nonnil-coherent ring. By \cite[Theorem 2.4]{aa16}, the class of $\phi$-flat modules is closed under direct products. Note that any pure submodule of a $\phi$-flat module is $\phi$-flat. Thus  the class of $\phi$-flat modules is preenveloping by \cite[Lemma 5.3.12, Corollary 6.2.2]{EJ00}. On the other hand, let $\{F_i\}_{i\in I}$ be a family of $\phi$-flat modules. Let $\prod_{i\in I} F_i\rightarrow F$ is a $\phi$-flat preenvelope. Then there is a factorization $\prod_{i\in I} F_i\rightarrow F\rightarrow F_i$ for each $i\in I$. Consequently, the natural composition $\prod_{i\in I} F_i\rightarrow F\rightarrow \prod_{i\in I} F_i$ is an identity. Thus $\prod_{i\in I} F_i$ is a direct summand of  $F$ and then $\prod_{i\in I} F_i$ is $\phi$-flat. It follows from \cite[Theorem 2.4]{aa16} that $R$ is  nonnil-coherent.

\end{proof}

The following Corollary follows from  Theorem \ref{asfap} and \cite[Corollary 6.3.5]{EJ00}.
\begin{corollary}\label{208}
Let $R$ be a  nonnil-coherent ring. If the class of $\phi$-flat modules  is closed under inverse limits, then the class of $\phi$-flat modules is enveloping.
\end{corollary}

\section{$\phi$-Dedekind rings and $\phi$-\Prufer\ rings }

Recall that an $R$-module $E$ is said to be  \emph{divisible} if $sM=M$ for any regular element $s\in R$, and an $R$-module $M$ is said to be  \emph{$h$-divisible} provided that $M$ is a quotient of an injective module. Evidently, any injective module is $h$-divisible and any $h$-divisible module is divisible. It is well known that an integral domain $R$ is a Dedekind domain if and only if any $h$-divisible module is injective, if and only if any divisible module is injective (see \cite[Theorem 5.2.15]{fk16} for example).
\begin{definition} Let $R$ be an $\NP$-ring. An $R$-module $E$ is called  \emph{nonnil-divisible} provided that  for any $m\in E$ and any non-nilpotent element $a\in R$, there exists  $x\in E$ such that $ax=m$.
\end{definition}

\begin{lemma}\label{nonnil-div-ext}
Let $R$ be an $\NP$-ring and $E$ an $R$-module. Consider the following statements:
 \begin{enumerate}
   \item $E$ is nonnil-divisible;
   \item $E$ is divisible;
  \item $\Ext_R^1(R/\langle a\rangle,E)=0$ for any $a\not\in\Nil(R)$.
 \end{enumerate}
Then we have $(1)\Rightarrow (2)$ and $(1)\Rightarrow (3)$. Moreover, if $R$ is a $\ZN$-ring, all  statements are equivalent.
\end{lemma}
\begin{proof} $(1)\Rightarrow (2)$ and $(2)\Rightarrow (1)$ for $\ZN$-rings: Trivial.

$(1)\Rightarrow (3)$: Let $a$ be a non-nilpotent element (then regular) in $R$ and $f: \langle a\rangle \rightarrow E$ be an $R$-homomorphism. Then there exists an element $x\in E$ such that $f(a)=ax$ since $E$ is nonnil-divisible. Set $g(r)=rx$ for any $r\in R$. Then $g$ is an extension of $f$ to $R$. Thus $\Ext_R^1(R/\langle a\rangle,E)=0$.

$(1)\Rightarrow (3)$ for $\ZN$-rings: Let $a$ be a non-nilpotent element in $R$ and  $m$  an element in $E$. Set $f(ra)=rm$. Then $f$ is a well-defined $R$-homomorphism from $\langle a\rangle$ to $E$. Since $\Ext_R^1(R/\langle a\rangle,E)=0$, there exists an $R$-homomorphism  $g: R\rightarrow E$ such that $g|_{\langle a\rangle}=f$. Let $x=g(1)$, then $m=f(a)=g(a)=ag(1)=ax$. Thus $E$ is  nonnil-divisible.
\end{proof}

The following result is an easy corollary of Lemma \ref{nonnil-div-ext}.
\begin{corollary}\label{nonnil-div}
Let $R$ be  a $\ZN$-ring and $E$ a nonnil-FP-injective $R$-module.  Then  $E$ is a nonnil-divisible  $R$-module.
\end{corollary}

\begin{lemma}\label{local nonnil-div}
Let $R$ be an $\NP$-ring and $E$ a nonnil-divisible  $R$-module. Then $E_{\p}$ is a nonnil-divisible  $R_{\p}$-module for any  prime ideal $\p$ of $R$.
\end{lemma}
\begin{proof}  Suppose $E$ is a nonnil-divisible  $R$-module. Let $\frac{m}{s}$ be an element in $ E_{\p}$ and $\frac{r}{t}$ a non-nilpotent element in $ R_{\p}$. Then $s,\ t$ and $r$ are non-nilpotent elements in $R$. Thus there exists $y\in E$ such that $tm=sry$ in $R$. Then $\frac{m}{s}=\frac{r}{t}\frac{y}{1}$. It follows that $E_{\p}$ is a nonnil-divisible  $R_{\p}$-module.
\end{proof}

Recall from \cite{FA04} that a $\phi$-ring $R$ is called a \emph{$\phi$-chained ring} if every $x\in R_{\Nil(R)}-\phi(R)$, we have $x^{-1}\in \phi(R)$, equivalently, if for any $a,b\in R-\Nil(R)$, either $a|b$ or $b|a$ in $R$. Moreover,  a $\phi$-ring $R$ is said to be a \emph{discrete $\phi$-chained ring} if $R$ is a $\phi$-chained ring with at most one nonnil prime ideal and every nonnil ideal of $R$ is principal (see \cite{FA05}).

\begin{proposition}\label{div-nonnil}
Let $R$ be a  discrete $\phi$-chained ring and $E$ a nonnil-divisible $R$-module. Then  $E$ is a nonnil-injective $R$-module.
\end{proposition}
\begin{proof} Let $I$ be a nonnil ideal of $R$. Since  $R$ is a discrete $\phi$-chained ring, then $I$ is generated by a non-nilpotent element $a\in R$.  Let $f:I\rightarrow E$ be an $R$-homomorphism. Then there exists $x\in E$ such that $f(a)=ax$ as $E$ is divisible. Define $g: R\rightarrow E$ by $g(r)=rx$. Then $g$ is an extension of $f$ to $R$. Hence $E$ is a nonnil-injective $R$-module.

\end{proof}

Recall that a regular ideal $I$ of $R$ is called \emph{invertible} if $II^{-1}=R$ where $I^{-1}=\{x\in\T(R)|Ix\subseteq R\}$. It follows from
\cite[Lemma 18.1]{H88} and \cite[Lemma 5.3]{GFT16} that a regular ideal is  invertible if and only if it is finitely generated and locally principal, if and only if it is projective.
Recall from \cite{FA04} that a nonnil ideal $I$ of a $\phi$-ring $R$ is said to be  \emph{$\phi$-invertible} provided that $\phi(I)$ is an invertible ideal of $\phi(R)$.
\begin{proposition}\label{pro-inv}
Let $R$ be a $\phi$-ring and $I$ a nonnil ideal of $R$. If $I$ is  projective over $R$, then $I$ is $\phi$-invertible.
\end{proposition}
\begin{proof} Since $I$ is a projective $R$-ideal,  $I$ is a direct summand of a free $R$-module $R^{(\kappa)}$. Then $\phi(I)$ is is a direct summand of a free $\phi(R)$-module $\phi(R)^{(\kappa)}$. Thus $\phi(I)$ is a projective $\phi(R)$-ideal.  Since $I$ is a nonnil ideal of $R$,
$\phi(I)$ is a regular ideal of  $\phi(R)$ by Lemma \ref{nonnil-regu}. By \cite[Lemma 5.3]{GFT16}, $\phi(I)$ is an invertible ideal of  $\phi(R)$. Thus  $I$ is $\phi$-invertible.
\end{proof}

Recall that an integral domain $R$ is a Dedekind domain if any nonzero ideal is invertible. Utilizing $\phi$-invertible, the authors in \cite{FA05} introduce $\phi$-Dedekind rings  which are  generalizations of  Dedekind domains to the context of rings that are in the class $\mathcal{H}$.
\begin{definition}
A $\phi$-ring $R$ is called \emph{$\phi$-Dedekind} provided that any nonnil ideal of $R$ is $\phi$-invertible.
\end{definition}

\begin{theorem}\label{asfap}
Let $R$ be a $\phi$-ring. Then the following statements are equivalent for $R$:
\begin{enumerate}
   \item $R$ is a $\phi$-Dedekind ring and a strong $\phi$-ring;

   \item   any  divisible module is nonnil-injective;

      \item   any $h$-divisible module is nonnil-injective;

      \item  any nonnil ideal of $R$ is projecitve.

\end{enumerate}
\end{theorem}
\begin{proof}  $(1)\Rightarrow(2)$: Let $E$ be a divisible module and $I$ a nonnil ideal of $R$. By \cite[Theorem 2.10]{FA05}, $R$ is nonnil-Noetherian. Then $I$ is finitely generated, and thus $R/I$ is finitely presented. Let $\m$ be a maximal ideal of $R$. Then $E_{\m}$ is a divisible module over $R_{\m}$ by Lemma \ref{nonnil-div-ext} and Lemma \ref{local nonnil-div}. By \cite[Theorem 2.10]{FA05} again, $R_{\m}$ is a discrete $\phi$-chained ring, thus $E_{\m}$ is a nonnil-injective $R_{\m}$-module by Proposition \ref{div-nonnil}. By \cite[Theorem 3.9.11]{fk16}, $\Ext_R^1(R/I,E)_{\m}=\Ext_{R_{\m}}^1(R_{\m}/I_{\m},E_{\m})=0$. Thus $\Ext_R^1(R/I,E)=0$. Therefore, $E$ is nonnil-injective.

$(2)\Rightarrow (3)$: Trival.

$(3)\Rightarrow(4)$: Let $N$ be an $R$-module, $I$ a nonnil ideal of $R$. There exists a long exact sequence as follows: $$0=\Ext_R^1(R,N)\rightarrow \Ext_R^1(I,N)\rightarrow \Ext_R^2(R/I,N)\rightarrow \Ext_R^2(R,N)=0.$$
Let $0\rightarrow N\rightarrow E\rightarrow K\rightarrow 0$ be an exact sequence where $E$ is the  injective envelope of $N$. There exists a long exact sequence as follows:
$$0=\Ext_R^1(R/I,E)\rightarrow \Ext_R^1(R/I,K)\rightarrow \Ext_R^2(R/I,N)\rightarrow \Ext_R^2(R/I,E)=0.$$
Thus $\Ext_R^1(I,N)\cong \Ext_R^2(R/I,N)\cong \Ext_R^1(R/I,K)=0$ as $K$ is nonnil-injective. It follows that $I$ is a projective ideal of $R$.

$(4)\Rightarrow(1)$:  It follows from Proposition \ref{pro-inv} that we just need to show $R$ is a strong $\phi$-ring.  Indeed, Let $a$ be non-nilpotent element in $R$. Then $\langle a\rangle$ is a projective ideal of $R$. It follows \cite[Corollary 2.6]{KMO22}  that $R$ is a strong $\phi$-ring.
\end{proof}

The next example shows that  every divisible module is not necessary nonnil-injective for  $\phi$-Dedekind rings. Thus the condition that $R$ is a strong $\phi$-ring in Theorem  \ref{asfap} cannot be removed.
\begin{example}\label{not div-de}
Let $D$ be non-field Dedekind domain   and $K$  its quotient field. Let $R=D(+)K/D$ be the idealization construction. Then $\Nil(R)=0(+)K/D$. Since $D\cong R/\Nil(R)$ is a Dedekind domain, $R$ is a  $\phi$-Dedekind ring by \cite[Theorem 2.5]{FA05}.   Denote by $U(R)$ and $U(D)$ the sets of unit elements of $R$ and $D$ respectively. Since $\Z(R)=\{(r,m)|r\in \Z(D)\cup \Z(K/D)\}=[R-\U(D)](+)K/D=R-\U(R)$ by \cite[Theorem 3.5, Theorem 3.7]{DW09}, $R$ is a total ring of quotient. Thus any $R$-module is divisible. However, since $\Nil(R)$ is not a maximal ideal of $R$, there exists an $R$-module $M$ which is not nonnil-injective by Theorem \ref{asfap-vn}.
\end{example}

Recall that an integral domain $R$ is a \Prufer\ domain if any finitely generated nonzero ideal is invertible. The following definition is a generalization of  \Prufer\ domains to the context of rings that are in the class $\mathcal{H}$ (see  \cite{FA04}).
\begin{definition}
A $\phi$-ring $R$ is called \emph{$\phi$-Pr\"{u}fer}  provided that any finitely generated nonnil ideal of $R$ is $\phi$-invertible.
\end{definition}

\begin{lemma}\label{w-phi-NN}
Let $R$ be an $\NP$-ring, $\p$  a  prime ideal of $R$ and $I$ an ideal of $R$. Then $I$ is nonnil over $R$ if and only if $I_{\p}$ is nonnil over $R_{\p}$.
\end{lemma}
\begin{proof}
Let $I$ be nonnil over $R$ and $x$ a non-nilpotent element in $I$. We will show the element $x/1$ in $I_{\p}$ is  non-nilpotent in $R_{\p}$. If $(x/1)^n=x^n/1=0$ in $R_{\p}$ for some positive integer $n$,  there is an $s\in R-\p$ such that $sx^n=0$ in $R$. Since $R$ is an $\NP$-ring,  $\Nil(R)$ is the minimal prime ideal of $R$. In the integral domain $R/\Nil(R)$, we have $\overline{s}\overline{x^n}=\overline{0}$, thus $\overline{x^n}=\overline{0}$ since $s\not\in  \Nil(R)$. So $x\in \Nil(R)$, a contradiction.

Let $x/s$ be a non-nilpotent element in $I_{\p}$ where $x\in I$ and $s\in R-{\p}$. Clearly, $x$ is non-nilpotent in $R$ and thus $I$ is nonnil over $R$.
\end{proof}

\begin{proposition}\label{w-phi-tor}
Let $R$ be an $\NP$-ring, $\p$  a  prime ideal of $R$ and $M$ an $R$-module. Then $M$ is $\phi$-torsion over $R$ if and only $M_{\p}$ is $\phi$-torsion over $R_{\p}$.
\end{proposition}
\begin{proof} Let $M$ be an $R$-module and $x\in M$. If $M_{\p}$ is $\phi$-torsion over $R_{\p}$, there is a nonnil ideal $I_{\p}$ over $R_{\p}$ such that $I_{\p} x/1=0$ in $R_{\p}$. Let $I$ be the preimage of $I_{\p}$ in $R$. Then $I$ is nonnil  by Lemma \ref{w-phi-NN}. Thus there is a non-nilpotent element $t\in I$  such that $tkx=0$ for some $k\not\in m$. Let $s=tk$. Then we have $\langle s\rangle$ is nonnil and $\langle s\rangle x=0$. Thus $M$ is $\phi$-torsion. Suppose $M$ is $\phi$-torsion over $R$.  Let $x/s$ be an element in $M_{\p}$. Then there is a nonnil ideal $I$ such that $Ix=0$, and thus $I_{\p}x/s=0$ with $I_{\p}\in\Nil(R_{\p})$  by Lemma \ref{w-phi-NN}. It follows that $M_{\p}$ is $\phi$-torsion over $R_{\p}$.
\end{proof}

\begin{theorem}\label{asfap-prufer}
Let $R$ be a $\phi$-ring. Then the following statements are equivalent for $R$:
\begin{enumerate}
   \item $R$ is a $\phi$-\Prufer\ ring and a strong $\phi$-ring.;

   \item   any  divisible module is nonnil-FP-injective;

      \item   any $h$-divisible module is nonnil-FP-injective;

      \item  any finitely generated nonnil ideal of $R$ is projecitve;

         \item  any $($finitely generated$)$ nonnil ideal of $R$ is flat;

       \item  any  $($finitely generated$)$  ideal of $R$ is $\phi$-flat;

       \item any submodule of $\phi$-flat module is $\phi$-flat;

       \item any $R$-module has an epimorphism $\phi$-flat preenvelope;

      \item any $R$-module has an epimorphism $\phi$-flat envelope.
      \end{enumerate}
\end{theorem}

\begin{proof}  $(1)\Rightarrow(2)$: Let  $T$ be a finitely presented $\phi$-torsion module and $\m$  a maximal ideal of $R$. Then by Proposition \ref{w-phi-tor}, $T_{\m}$ is a finitely presented $\phi$-torsion $R_{\m}$-module.
By \cite[Corollary 2.10]{FA04}, $R_{\m}$ is a $\phi$-chained ring. Since $R$ is a strong $\phi$-ring, $R_{\m}$ is a strong $\phi$-ring.    Thus  $T_{\m}\cong \oplus_{i=1}^n R_{\m}/R_{\m}x_i$ for some regular element $x_i\in R_{\m}$ by \cite[Theorem 4.1]{Z18}. Let $E$ be a divisible module. Then $E_{\m}$ is a divisible module over $R_{\m}$ by Lemma \ref{nonnil-div-ext} and Lemma \ref{local nonnil-div}.
 Thus $\Ext_R^1(T,E)_{\m}=\Ext_{R_{\m}}^1(T_{\m},E_{\m})=\oplus_{i=1}^n\Ext_{R_{\m}}^1( R_{\m}/R_{\m}x_i, E_{\m})=0$ by  Lemma \ref{nonnil-div-ext} and \cite[Theorem 3.9.11]{fk16}. It follows that $\Ext_R^1(T,E)=0$. Therefore, $E$ is nonnil-FP-injective.

$(2)\Rightarrow  (3)$: Trival.

$(3)\Rightarrow(4)$: Let $N$ be an $R$-module, $I$ a finitely generated nonnil ideal of $R$. The short exact sequence $0\rightarrow I\rightarrow R \rightarrow R/I\rightarrow 0$ induces a long exact sequence as follows: $$0=\Ext_R^1(R,N)\rightarrow \Ext_R^1(I,N)\rightarrow \Ext_R^2(R/I,N)\rightarrow \Ext_R^2(R,N)=0$$
Let $0\rightarrow N\rightarrow E\rightarrow K\rightarrow 0$ be an exact sequence where $E$ is the  injective envelope of $N$. There exists a long exact sequence as follows:
$$0=\Ext_R^1(R/I,E)\rightarrow \Ext_R^1(R/I,K)\rightarrow \Ext_R^2(R/I,N)\rightarrow \Ext_R^2(R/I,E)=0$$
Thus $\Ext_R^1(I,N)\cong \Ext_R^2(R/I,N)\cong \Ext_R^1(R/I,K)=0$ as $K$ is nonnil-FP-injective. It follows that $I$ is a projective ideal of $R$.

$(4)\Rightarrow(1)$: It follows from Proposition \ref{pro-inv} and Theorem \ref{asfap}.

$(4)\Rightarrow(5)$: Let $I$ be nonnil ideal of $R$ and $a$ a non-nilpotent element in $I$. Let $\{I_i\}_{i\in \Gamma}$ be a family of finitely generated sub-ideal of $I$ such that $\lim\limits_{\longrightarrow }I_i=I$. Set $I'_i=\langle I,a\rangle$,  then $I'_i$ is a finitely generated nonnil ideal of $R$ such that $\lim\limits_{\longrightarrow }I'_i=I$. Since each $I'_i$ is projective by (4), $I$ is a flat ideal of $R$.

$(5)\Leftrightarrow(6)$:  Let $I$ be a (resp., finitely generated) nonnil ideal of $R$ , $J$ an (resp., a finitely generated) ideal of $R$. Then we have $\Tor_1^R(R/J, I)\cong\Tor_2^R(R/I, R/J)\cong\Tor_1^R(R/I, J)$. Thus $I$ is flat if and only if $J$ is $\phi$-flat.

$(5)\Rightarrow(1)$: It follows \cite[Corollary 2.6]{KMO22}  that $R$ is a strong $\phi$-ring. Let $K,\ L$ be non-zero (resp., finitely generated) ideals of $R/\Nil(R)$ $($denoted by $\overline{R})$. Then $K=I/\Nil(R)$ and $L=J/\Nil(R)$ for some (resp., finitely generated) nonnil ideals $I,\ J$ of $R$ (see \cite[Lemma 2.4]{FA04}). By \cite[Lemma 1.6]{ZxlZ20}, $J\Nil(R)=\Nil(R)$. Thus $L=J/\Nil(R)\cong J\otimes_R \overline{R}$ .
We claim that $\Tor_1^{\overline{R}}(\overline{R}/K,L)=0$. Indeed, by change of rings, the exact sequence of  $\overline{R}$-modules: $$0\rightarrow \Tor_1^{\overline{R}}(\overline{R}/K,J\otimes_R \overline{R}) \rightarrow K\otimes_{\overline{R}}J\otimes_R\overline{R}\rightarrow  \overline{R}\otimes_{\overline{R}}J\otimes_R\overline{R} \rightarrow  \overline{R}/K\otimes_{\overline{R}}J\otimes_R\overline{R}\rightarrow 0$$
is naturally isomorphic to
$$0\rightarrow \Tor_1^{\overline{R}}(\overline{R}/K,J\otimes_R \overline{R}) \rightarrow K\otimes_{R}J\rightarrow  \overline{R}\otimes_{R}J \rightarrow  \overline{R}/K\otimes_{R}J\rightarrow 0.$$
Thus there is a commutative diagram of $R$-modules:
$$\xymatrix@R=25pt@C=20pt{
0\ar[r]^{} &  \Tor_1^R(R/I,J)  \ar[d]_{f}\ar[r]^{}               & I\otimes_R J\ar[d]_{g}\ar[r]^{}      & R\otimes_R J \ar[d]^{h}\ar[r]^{} & R/I\otimes_RJ\ar[r]^{}\ar[d]^{\cong}& 0\\
0\ar[r]^{} & \Tor_1^{\overline{R}}(\overline{R}/K,J\otimes_R\overline{R}) \ar[r]^{} &    K\otimes_{R}J \ar[r]^{}   &\overline{R}\otimes_{R}J \ar[r]^{}   & \overline{R}/K\otimes_{R}J \ar[r]^{}    &  0.\\}$$
Since $g$ and $h$ are epimorphisms, $f$ is also an epimorphism by the Five Lemma (see \cite[Theorem 1.9.9]{fk16}). By (5) $J$ is flat, then $\Tor_1^{R}(R/I,J)=0$. Thus $\Tor_1^{\overline{R}}(\overline{R}/K,L) \cong \Tor_1^{\overline{R}}(\overline{R}/K,J\otimes_R \overline{R})=0$. Consequently, $\overline{R}=R/\Nil(R)$ is a \Prufer\ domain. By \cite[Corollary 2.10]{FA04}), $R$ is a $\phi$-\Prufer\ ring.

$(5)\Rightarrow (7)$:  Let $M$ be a $\phi$-flat module and $N$ a submodule of $M$. Let $I$ be a nonnil ideal of $R$, then $I$ is flat by $(6)$. Thus $\fd_R(R/I)\leq 1$. By considering the long exact sequence $\Tor_2^R(R/I,M/N)\rightarrow \Tor_1^R(R/I,N)\rightarrow \Tor_1^R(R/I,M) $, we have $\Tor_1^R(R/I,N)=0$ as $\Tor_2^R(R/I,M/N)=\Tor_1^R(R/I,M) =0$. Thus $N$ is $\phi$-flat.

 $(7)\Rightarrow (6)$ and $(9)\Rightarrow (8)$ : Trivial.

$(8)\Rightarrow (7)$: Let $F$ be a $\phi$-flat module, $i:K\rightarrowtail F$ a monomorphism and $f:K\twoheadrightarrow F'$ an epimorphism $\phi$-flat preenvelope. Then there exists an homomorphism $g:F'\rightarrow F$ such that $i=gf$. Thus $f$ is  a monomorphism. Cosequently,  $K\cong F'$ is $\phi$-flat.

$(1)+(4)+(7)\Rightarrow (9)$: Let $R$ be a $\phi$-\Prufer\ ring and $I$  a finitely generated nonnil ideal of $R$. By (4), $I$ is projective and thus finitely presented. It follows that $R$ is nonnil-coherent. Thus the class of $\phi$-flat modules is preenveloping by Proposition \ref{asfap-nc}.   Let $\{F_i|i\in I\}$ be a family of $\phi$-flat modules.  Then $\prod_{i\in I} F_i$ is $\phi$-flat by \cite[Theorem 2.4]{aa16}. By (7), the class of $\phi$-flat modules is closed under submodules. Thus the class of $\phi$-flat modules is closed under inverse limits. By corollary \ref{208}, the class of $\phi$-flat modules is enveloping.

We claim that the $\phi$-flat envelope of any $R$-module $M$ is an epimorphism. Indeed,  suppose $f: M\rightarrow F$ be $\phi$-flat envelope of $M$. Let $f=h\circ g$ with $g:M\twoheadrightarrow \Im{f}$  an epimorphism and $f:\Im{f}\rightarrowtail F$ the embedding map. We will show $g$ is the the $\phi$-flat envelope of $M$. For any  $f': M\rightarrow F'$ with $F'$ $\phi$-flat, there exists $l:F\rightarrow F'$ such that  $l\circ f=f'$. Then $g\circ h\circ l=f'$, and thus  $g$ is a $\phi$-flat  preenvelope of $M$ as $\Im{f}$ is $\phi$-flat by (7).  Suppose $a:\Im{f}\rightarrow\Im{f}$ such that $g=a\circ g$. Then $a$ is an epimorphism. Consider the following commutative diagram:
$$\xymatrix@R=25pt@C=25pt{
  M \ar@{->>}[r]^{g}\ar@{=}[d] & \Im f \ar[d]_{a}\ar@{>->}[r]^{h} & F\ar@{.>}[d]_{b} \\
  M \ar@{->>}[r]^{g} &\Im f \ar@{>->}[r]^{h} & F\\
}$$
Since $f=h\circ g$ is an $\phi$-flat envelope, there exists $b:F\rightarrow F$ such that $b\circ f=b\circ h\circ g=h\circ a\circ g=h\circ g=f$. Since $g$ is an epimorphism, $h\circ a=b\circ h$.
Then $a$ is a monomorphism, and thus  $a$  is an isomorphism. It follows that $g$ is the $\phi$-flat envelope of $M$.
\end{proof}

\begin{remark}\label{not phi-tf-phi-flat}
Actually, Zhao \cite[Theorem 4.3]{Z18} showed that if $R$ is a strong $\phi$-ring, then $R$ is a $\phi$-\Prufer\ ring if and only if each submodule of a $\phi$-flat $R$-module is $\phi$-flat, if and only if each nonnil ideal of $R$ is $\phi$-flat, if and only if finitely generated nonnil ideal of $R$ is $\phi$-flat. In Theorem \ref{asfap-prufer}, we give some simple versions of \cite[Theorem 4.3]{Z18}  and several new characterizations $\phi$-\Prufer\ ring using divisible modules, nonnil-FP-injective modules and  the epimorphic enveloping properties of  $\phi$-flat modules.
\end{remark}

The final example shows that  every divisible $R$-module is not necessary nonnil-FP-injective for $\phi$-\Prufer\ rings. Thus the condition that $R$ is a strong $\phi$-ring in Theorem \ref{asfap-prufer} also cannot  be removed.
\begin{example}\label{not div-de-pru}
Let $D$ be  non-field \Prufer\ domain and $K$  its quotient field. Let $R=D(+)K/D$ be the idealization construction. As in Example  \ref{not div-de}, we can show $R$ is a $\phi$-\Prufer\ ring and total ring of quotient. Thus any $R$-module is divisible. However, since $\Nil(R)$ is not a maximal ideal of $R$, the Krull dimension of $R>1$. Thus there exists an $R$-module $M$ which is not nonnil-FP-injective by Theorem \ref{asfap-vn}.
\end{example}

\begin{acknowledgement}\quad\\
The first author was supported by the Natural Science Foundation of Chengdu Aeronautic Polytechnic (No. 062026) and the National Natural Science Foundation of China (No. 12061001).
\end{acknowledgement}

\end{document}